\documentclass[fleqn,12pt]{SelfArx} 

\usepackage[english]{babel} 
\usepackage{subfigure}

\definecolor{color1}{RGB}{25,25,112} 
\definecolor{color2}{RGB}{10,0,100} 

\usepackage{hyperref} 
\hypersetup{hidelinks,colorlinks,breaklinks=true,urlcolor=color2,citecolor=color1,linkcolor=color1,bookmarksopen=false,pdftitle={Title},pdfauthor={Author}}

\JournalInfo{$\;$} 
\Archive{January 23rd, 2018} 
\title{Non-polyhedral extensions of the Frank-and-Wolfe theorem}

\PaperTitle{Non-polyhedral extensions of the Frank-and-Wolfe theorem} 

\Authors{Juan Enrique Mart\'inez-Legaz\textsuperscript{1}*, Dominikus Noll\textsuperscript{2}, Wilfredo Sosa\textsuperscript{3}} 
\affiliation{\textsuperscript{1}\textit{Departament d'Economia i d'Hist\`oria Econ\`omica, Universitat Aut\`onoma de Barcelona, Spain}} 
\affiliation{\textsuperscript{2}\textit{Institut de Math\'ematiques de Toulouse, 118 route de Narbonne 31062, Toulouse, France}} 
\affiliation{\textsuperscript{3}\textit{Programa de P\^os-Gradua\c{c}\~ao em Economia, Universidade Cat\'olica de Bras\'ilia, Brazil}}
\affiliation{*\textbf{Corresponding authors}:  JuanEnrique.Martinez.Legaz@uab.cat \& dominikus.noll@math.univ-toulouse.fr} 

\Keywords{Quadratic optimization --- asymptotes --- Motzkin-sets --- Frank-and-Wolfe theorem 
\vspace*{.1cm} 
} 
\newcommand{\keywordname}{Keywords} 

 \Abstract{In 1956 Marguerite Frank and Paul Wolfe proved that a quadratic function which is bounded below on a polyhedron 
$P$ attains its infimum on $P$. In this work we search for larger classes of sets $F$ with this Frank-and-Wolfe property. 
We establish the existence of non-polyhedral Frank-and-Wolfe sets, obtain internal characterizations by way of asymptotic 
properties, and investigate stability of the Frank-and-Wolfe class under various operations.}

\usepackage{algorithmic}
\usepackage[justification=centering]{caption}

\usepackage{amsmath, amsthm, amssymb}
\usepackage{multirow, array}
\usepackage{float}

\usepackage{mathrsfs}

\usepackage{pst-plot}
 \usepackage{pdftricks}
 \begin{psinputs}
 \usepackage{pstricks,pst-plot}
\end{psinputs}

\theoremstyle{definition}
\newtheorem{definition}{Definition}
\newtheorem{remark}{Remark}
\newtheorem{lemma}{\bf Lemma}
\newtheorem{proposition}{\bf Proposition}
\newtheorem{theorem}{\bf Theorem}
\newtheorem{corollary}{\bf Corollary}
\newtheorem{example}{\bf Example}{\rm}

\usepackage{algorithm}

\makeatletter
\renewcommand{\p@algorithm}{\arabic{algorithm}\expandafter\@gobble}
\makeatother

\newcounter{step}[algorithm]
\newcommand\STEP[2][\(\triangleright\)]{%
	\refstepcounter{step}
	\vskip 0.25\baselineskip
	\item[]\hskip -\algorithmicindent #1 \textbf{Step \arabic{step}}%
	\ifthenelse{\equal{\unexpanded{#2}}{}}{}{ (\texttt{#2})}%
	\textbf{.}%
}
\newenvironment{algo}{\algo}{}
\def\algo#1\end{%
	\noindent\fbox{%
	\begin{minipage}[b]{\dimexpr\columnwidth-\algorithmicindent\relax}
	\begin{algorithmic}
	#1
	\end{algorithmic}
	\end{minipage}
	}%
\end}

\newcommand{\ttop}{\mathsf{T}}

\begin{document}

\maketitle

\section{Introduction}
\label{sec:intro}
In this paper we investigate extensions of the famous Frank and Wolfe theorem
\cite{frank_and_wolfe,oettli,collatz,eaves,andronov}, which states that a quadratic function $f$ which is bounded below on a closed convex polyhedron 
$P$ attains its infimum on $P$. This has applications to linear complementarity 
problems, and a natural question is whether this property is shared by larger classes
of non-polyhedral convex sets $F$. 

The present work expands on \cite{JCA}, where the Frank-and-Wolfe property was successfully related
to asymptotic properties of a set $F$. 
Following this line, we presently obtain a complete
characterization of the Frank-and-Wolfe property within the class of Motzkin decomposable
sets.  In particular, the converse of a result of Kummer \cite{kummer} is obtained.


A second theme addresses versions of the Frank-and-Wolfe theorem where the class of quadratic functions
is further restricted. One may for instance ask for  sets $F$ on which convex or quasi-convex quadratics attain their finite infima.
It turns out that this class has a complete characterization as those sets which have no flat asymptotes in the sense of Klee.
 As a consequence we obtain a version of the Frank-and-Wolfe theorem which extends a result of Rockafellar \cite[Sect. 27]{rock} 
 and Belousov and Klatte \cite{klatte} on
 convex polynomials.

Invariance of Frank-and-Wolfe type sets under various operations such as finite
intersections, unions, cross-products, sums, and under affine images and pre-images are also investigated.

The structure of the chapter is as follows. In section
\ref{sec:FWsets} we give the definition and collect basic information on {\em FW}-sets.
In section \ref{qFW} we consider quasi-Frank-and-Wolfe sets, where a version
of the Frank and Wolfe theorem for quasi-convex quadratics is discussed.
It turns out that the same class allows  many more applications, 
as it basically suffices to have polynomial functions which have at least one 
convex sub-level set. In section \ref{sec:Motzkin} we consider sets with a generalized Motzkin decomposition of the form
$F = K+D$ with $K$ compact and $D$ a closed convex cone. This class was used
by Kummer \cite{kummer}, who proved a version of the Frank and Wolfe theorem in this class when $D$ is polyhedral.
We give a new proof of this result and also establish its inverse, that is, if
a Motzkin set satisfies the Frank and Wolfe theorem, then the cone $D$
must be polyhedral. Section \ref{invariance} discusses invariance properties
of the class of Motzkin sets with the Frank and Wolfe property.

\section*{Notations}

We generally follow Rockafellar's book \cite{rock}. The closure of a set $F$
is $\overline{F}$. The Euclidean norm in $\mathbb{R}^n$ is $\|\cdot\|$, and
the Euclidean distance is $\mathrm{dist}(x,y)=\|x-y\|$. For subsets $M,N$ of 
$\mathbb{R}^n$ we write $\mathrm{dist}(M,N)=\inf\{\|x-y\|: x\in M, y\in N\}$%
. A direction $d$ with $x+td\in F$ for every $x\in F$ and every $t\geq 0$ is
called a direction of recession of $F$, and the cone of all directions of
recession is denoted as $0^+F$.

A function $f(x) = \frac{1}{2} x^\mathsf{T} Ax + b^\mathsf{T} x + c$ with $A
= A^\mathsf{T}\in \mathbb{R}^{n\times n}$, $b\in \mathbb{R}^n$, $c\in 
\mathbb{R}$ is called quadratic. The quadratic $f:\mathbb{R}^n\to \mathbb{R}$
is quasi-convex on a convex set $F\subset \mathbb{R}^n$ if the sub-level sets
of $f_{|F}:F\to\mathbb{R}$ are convex. Similarly, $f$ is convex
on the set $F$ if $f_{|F}$ is convex.

\section{Frank-and-Wolfe sets}
\label{sec:FWsets}
The following definition is the basis for our investigation:

\begin{definition}
A set $F \subset \mathbb R^n$ is called a  {\em Frank-and-Wolfe set}, for short
a {\em FW}-set, if every quadratic function $f$ which is bounded below on $F$
attains its infimum on $F$. 
\end{definition}

In \cite{JCA} this notion was introduced for convex sets $F$,
but in the present note we extend it to arbitrary sets, as this property is not really related to convexity. 
The classical Frank-and-Wolfe theorem says that every closed convex polyhedron is
a {\em FW}-set, cf. \cite{frank_and_wolfe,oettli,collatz,eaves}. Here we are interested in identifying and characterizing more general classes of sets with this property. 
We start by collecting some basic information about  {\em FW}-sets.

\begin{proposition}
\label{affine}
{\em
Affine images of FW-sets are again FW-sets. }
\end{proposition}

\begin{proof}
Let $F$ be a {\em FW}-set in $\mathbb R^n$ and and $T:\mathbb R^n \to \mathbb R^m$ an affine
mapping. We have to show that $T(F)$ is a {\em FW}-set.
Let $f$ be a quadratic on $\mathbb R^m$ which is bounded below on $T(F)$, then $f \circ T$ is a quadratic
on $\mathbb R^n$, which is bounded below on $F$, hence attains its infimum at some $x\in F$. Then $f$ attains its infimum at
$Tx\in T(F)$.
\end{proof}

It is equally easy to see that every {\em FW}-set is closed, because if $x\in \overline{F}$, then
the quadratic function $f = \|\cdot - x\|^2$ has infimum 0 on $F$, and if this infimum is to be attained, then
$x\in F$. As a consequence, a
bounded set $F$ is {\em FW} iff it is closed, so 
there is nothing interesting to report on bounded {\em FW}-sets,
and the property is clearly aimed at the analysis of unbounded sets. 

One can go a little further than just proving closedness of {\em FW}-sets and get first information about their asymptotic
behavior. We need the following:

\begin{definition}
An affine manifold $M$ in $\mathbb R^n$ is called an $f$-{\em asymptote} of the set $F \subset \mathbb R^n$
if $F \cap M = \emptyset$ and dist$(F,M)=0$. 
\end{definition}

This 
expands on   Klee \cite{Kl60}, who introduced this notion for convex sets $F$. The symbol $f$ stands for {\em flat} asymptote.
This allows us now to
propose the following:

\begin{proposition}
\label{FWnof}    
{\em
Let $F$ be a FW-set. Then $F$ has no f-asymptotes.}
\end{proposition}

\begin{proof}
Let $M$ be an affine subspace such that 
dist$(F,M)=0$. We have to show that $M\cap F\not=\emptyset$.
Let $M = \{x\in \mathbb R^n: Ax-b=0\}$ for a suitable matrix $A$ and vector $b$.
Put $f(x) = \|Ax-b\|^2$, then $f$ is quadratic, and $\gamma = \inf\{f(x) : x \in F\} \geq 0$. Now there exist $x_n\in F$ and 
$y_n\in M$ with dist$(x_n,y_n)\to 0$. But $Ay_n=b$, and $\|A(x_n-y_n)\| \leq \|A\| \|x_n-y_n\|\to 0$,
hence $Ax_n \to b$, which implies $\gamma = 0$. Now since $F$ is a $FW$-set, this infimum is attained,
hence there exists $x\in F$ with $f(x)=0$, which means $Ax=b$, hence $x\in M$. That shows
$F \cap M\not= \emptyset$, so $M$ is not an $f$-asymptote of $F$.
\end{proof}

\begin{remark}
An immediate consequence of Propositions \ref{affine}, \ref{FWnof} is that
affine images of {\em FW}-sets, and in particular, projections of {\em FW}-sets, are always closed.
\end{remark}

Yet another trivial fact is the following:

\begin{proposition}
{\em 
Finite unions of {FW}-sets are {FW}. }
\hfill $\square$
\end{proposition}

We conclude this preparatory section by looking at invariance of the $FW$-class under affine pre-images.
First we need the following:

\begin{proposition}
\label{prodFW}
{\em
If $F\subset \mathbb{R}^{n}$ is a FW-set and $M\subset \mathbb{R}^{m}$ is
an affine manifold, then $F\times M$ is a FW-set in $\mathbb R^n \times \mathbb R^m$.}
\end{proposition}

\begin{proof}
Since translates of $FW$-sets are $FW$-sets,  
we may assume that $M$ is a linear subspace, and then   there is
no loss of generality in assuming that $M=\mathbb{R}^{m}.$  Moreover, by an easy
induction argument, we only need to consider the case when $m=1.$ 

Let $q$ be
a quadratic function on $\mathbb{R}^{n}\times \mathbb{R}$ bounded below on $%
F\times \mathbb{R}$. We can write $q\left( x,t\right) =\frac{1}{2}x^{T}Ax+%
\frac{1}{2}bt^{2}+tc^{T}x+d^{T}x+et+f$ for suitable $A,$ $b,$ $c,$ $d$, $e$
and $f.$ Clearly, $b\geq 0$, as otherwise $q$ could not be bounded below on $F \times \mathbb R$.  Now  we have $\inf_{\left( x,t\right) \in F\times 
\mathbb{R}}q\left( x,t\right) =\inf_{x\in F}\inf_{t\in \mathbb{R}}q\left(
x,t\right) .$ 

First consider the case
$b>0$.  Then the inner infimum in the preceding expression
is attained at $t=-\frac{c^{T}x+e}{b}.$     Hence we have $\inf_{\left(
x,t\right) \in F\times \mathbb{R}}q\left( x,t\right) =\inf_{x\in F}q\left(
x,-\frac{c^{T}x+e}{b}\right) .$ Given that $q\left( x,-\frac{c^{T}x+e}{b}%
\right) $ is a quadratic function of $x$ and is obviously bounded below on $%
F,$ it attains its infimum over $F$ at some $\overline{x}\in F.$ Therefore $q$
attains its infimum over $F\times \mathbb{R}$ at $\left( \overline{x},-\frac{%
c^{T}\overline{x}+e}{b}\right) .$ 

Now consider the case $b=0$, $c \not=0$.
Here $F$ must be contained in
the hyperplane $c^{T}x+e=0.$  Substituting this,  we get
$\inf_{\left( x,t\right) \in F\times \mathbb{R}}q\left( x,t\right)
=\inf_{x\in F}\left\{ \frac{1}{2}x^{T}Ax+d^{T}x\right\} +f.$ Hence, the
quadratic function given by $\frac{1}{2}x^{T}Ax+d^{T}x$ is bounded below on $%
F$ and, for every minimizer $\overline{x}\in F$ and every $t\in \mathbb{R,}$
the point $\left( \overline{x},t\right) $ is a minimizer of $q$ over $%
F\times \mathbb{R}$.

Finally, when $b=0$, $c=0$ it follows that we must also have $e=0$, so $q$ no longer depends on $t$,
and we argue as in the previous case.
\end{proof}

\begin{remark}
As we shall see in the next section (example 1), the cross product $F_1 \times F_2$  of two $FW$-sets $F_i$ is in general no longer
a $FW$-set, so Proposition \ref{prodFW} exploits the very particular situation.
\end{remark}

We have the following consequence:

\begin{corollary}
\label{cor1}
{\em
Let $F$ be a FW-set in $\mathbb R^n$ and $M$ an affine subspace of $\mathbb R^n$. Then $F+M$ is a FW-set.}
\end{corollary}

\begin{proof}
$F \times M$ is a {\em FW}-set by Proposition \ref{prodFW}, and its image
under the mapping $(x,y) \to x+y$ is a {\em FW}-set by Proposition \ref{affine}, and that set is $F+M$. 
\end{proof}

Concerning pre-images, we have the following consequence of Proposition \ref{prodFW}:

\begin{proposition}
\label{prop5}
{\em
Let $T$ be an affine operator and suppose the $FW$-set  $F$  is contained in the range of $T$.
Then $T^{-1}(F)$ is a $FW$-set.}
\end{proposition}

\begin{proof}
Since the notion of a $FW$-set is invariant under translations and under coordinate changes, we can assume that
$T$ is a  surjective linear operator $T:\mathbb R^n \to \mathbb R^m$ and $F \subset \mathbb R^m$. Now
$\widetilde{F} = (T_{|{\rm ker}(T)^\perp})^{-1}(F)$ is an affine image of the $FW$-set
$F$, hence by Proposition \ref{affine} is a $FW$-subset of $\mathbb R^n$. By Corollary \ref{cor1} the set $\widetilde{F}  + {\rm ker}(T)$ is a
$FW$-set, but this set is just $T^{-1}(F)$.
\end{proof}

{\color{black}
\begin{remark}
It is not clear whether this result remains true when $F$ is not entirely contained in the range of $T$, i.e., when
only $F \cap {\rm range}(T)\not=\emptyset$. In contrast, see Corollary \ref{cor4} and Proposition \ref{pre-image}.
\end{remark}
}

More sophisticated invariance 
properties of the class of {\em FW}-sets will be investigated
later. For instance, one may ask whether or under which conditions finite intersections, cartesian
products,  or closed subsets of {\em FW}-sets are again {\em FW}.
\\

%

\section{Frank-and-Wolfe theorems for restricted classes of quadratic functions}
\label{qFW}

Following \cite{JCA} it is also of interest to investigate
versions of the Frank and Wolfe theorem, where the class of quadratic functions is further
restricted. The following notion is from \cite{JCA}:

\begin{definition}
A convex set $F \subset \mathbb R^n$ is called a {\em quasi-Frank-and-Wolfe set},
for short a {\em qFW}-set,  if every quadratic function $f$,
which is quasi-convex on $F$ and bounded below on $F$, attains its infimum on $F$. 
\end{definition}

Note that
for the class of {\em qFW}-sets we have to maintain convexity as part of the definition,
as otherwise absurd situations might occur,
so the notion is precisely as introduced
in \cite{JCA}.

\begin{remark}
Every convex {\em FW}-set is clearly a {\em qFW}-set.  The converse is not true, i.e., {\em qFW}-sets
need not be {\em FW}, as will be seen in Example \ref{example_LZ}.  It is again clear that
{\em qFW}-sets are closed, and that affine images of {\em qFW}-sets
are {\em qFW}.  
\end{remark}

It turns out that $f$-asymptotes are the key to understanding
the quasi-Frank-and-Wolfe property. We have  the following:

\begin{theorem}
\label{asymptote} 
{\em Let $F$ be a convex set in $\mathbb{R}^n$. Then the
following statements are equivalent:
\begin{enumerate}
\item[\rm (1)] Every polynomial $f$ which has at least one nonempty convex sub-level set
on $F$ and which is bounded below on $F$ attains its infimum on $F$.
%
\item[\rm (2)]  $F$ is a qFW-set.
\item[\rm (3)] Every quadratic function $q$ which is convex on $F$ and bounded
below on $F$ attains its infimum on $F$. 
\item[\rm (4)] $F$  has no {f}-asymptotes.
\item[\rm (5)] $T(F)$ is closed for every affine mapping $T$.
\item[\rm (6)] $P(F)$ is closed for every orthogonal projection $P$.
\end{enumerate}
}
\end{theorem}

\begin{proof}
The implication (1) $\implies$ (2) is clear, because for a quasi-convex function on
$F$ {\em every} sub-level set on $F$ is convex. 
The implication (2) $\implies$ (3) is also evident. 
Implication (3) $\implies$ (4)  
follows immediately with the same proof as Proposition \ref{FWnof}, because
the quadratic $f(x)=\|Ax-b\|^2$ used there is convex.

Let us prove (4) $\implies$ (5).  
We may without loss of generality assume
that $T$ is linear, as properties (4) and (5) are invariant under translations.
Suppose $T(F)$ is not closed and pick $y \in \overline{T(F)}\setminus T(F)$.
Put $M = T^{-1}(y)$,   then $M$ is an affine manifold. Note that
$M \cap F = \emptyset$, because $T(M) =\{y\}$. Now pick $y_n \in T(F)$
such that $y_n \to y$, and choose $x_n   \in T^{-1}(y_n) \cap F$. Since $T$ is affine there exist
$x_n'\in T^{-1}(y_n)$ such that $x_n'\to x'\in T^{-1}(y)$. We have 
$\|x_n - (x' - x_n' + x_n)\| \to 0$, with $x_n \in F$,  and since
$x_n-x'_n\in {\rm ker}(T)$, we have $x' -x_n'+x_n \in x' + {\rm ker}(T) = M$.
That proves
dist$(F,M)=0$, and so $F$ has $M$ as  an $f$-asymptote, a contradiction.

The implication (5) $\implies$ (6) is clear.
Let us prove (6) $\implies$ (1). 
We will prove this by	induction on $n$. For $n=1$ the implication is clearly
true, because any polynomial $f:\mathbb{R}\rightarrow \mathbb{R}$
which is bounded below on a  convex set $F\subset \mathbb{R}$ satisfying (6) attains
its infimum on $F$, as (6) implies that $F$ is closed. Suppose therefore that the result is true for dimension $%
n-1$, and consider a polynomial $f:\mathbb{R}^{n}\rightarrow \mathbb{%
R}$ which is bounded below on a set $F\subset \mathbb R^n$ with property (6) such that $S_{\alpha }:=\{x\in F:f(x)\leq
\alpha \}$ is nonempty and convex for some $\alpha \in \mathbb{R}.$ We may without loss of generality
assume that the dimension of $F$ is $n$, i.e., that $F$ has nonempty interior, as
otherwise $F$ is contained in a hyperplane,
and then the result follows directly from the induction hypothesis. If $%
\alpha =\gamma :=\inf \{f(x):x\in F\},$ then $f$ clearly attains $\alpha ,$
so we assume from now on that $\alpha >\gamma $. If $S_{\alpha }:=\{x\in
F:f(x)\leq \alpha \}$ is bounded, then by the Weierstrass extreme value
theorem the infimum of $f$ over $S_{\alpha }$ is attained, because by hypothesis (6) the set  $F$ is closed. But this infimum
is also the infimum of $f$ over $F$, so in this case we are done. Assume
therefore that $S_{\alpha }$ is unbounded. Since $S_{\alpha }$ is a closed
convex set, it has a direction of recession $d$, that is, $x+td\in S_{\alpha
}$ for every $t\geq 0$ and every $x\in S_{\alpha }$. Fix $x\in S_{\alpha }$.
This  means
\[
\gamma \leq f(x+td) \leq \alpha
\]
for every $t\geq 0$. 
Since $t \mapsto f(x+td)$ is a polynomial on the real line, which is now bounded on 
$[0,\infty)$, it must be constant as a function of $t$, so that
$f(x) = f(x+td)$ for all $t\geq 0$, and then clearly also 
$f(x+td) = f(x)$ for every $t\in \mathbb R$.
But the argument
is valid for every $x\in S_{\alpha }$. By assumption $F$ has dimension $n$,
so $S_{\alpha }$ has nonempty interior. 
That shows $f(x+td) = f(x)$ for all $x$ in a nonempty open set contained in $S_\alpha$ 
and all $t\in \mathbb R$. 
Altogether, since $f$ is a polynomial,  we obtain%
\begin{equation}
f(x+td)=f(x)\text{ for every }x\in \mathbb{R}^{n}\text{ and every }t\in 
\mathbb{R}.  \label{same}
\end{equation}%

Now let $P$ be the orthogonal projection onto the hyperplane $H=d^{\perp }$.
Then $\widetilde{f}:=f_{|H}$ 
is a polynomial on the $(n-1)$%
-dimensional space $H$ and takes the same values as $f$ due to (\ref{same}).
In particular, $\widetilde{f}=f_{|H}$ 
is bounded below on the 
set $\widetilde{F}=P(F)$. 

We argue that the induction hypothesis applies to $\widetilde{F}$.
Indeed, $\widetilde{F}$  being the image of $F$ under a projection, is closed
by condition (6). Its dimension is $n-1$, and moreover, every
projection of $\widetilde{F}$ is closed, because
any such projection is also a projection of $F$.

It remains to prove that the
restriction of $\widetilde{f}$ to $\widetilde{F}$ has a nonempty convex
sub-level set. To this end it will suffice to prove that, for $\widetilde{S}%
_{\alpha }:=\{x\in \widetilde{F}:\widetilde{f}(x)\leq \alpha \},$ one has $%
\widetilde{S}_{\alpha }=P\left( S_{\alpha }\right) .$ This will easily
follow from the observation that $\widetilde{f}\circ P=f,$ which is an
immediate consequence of (\ref{same}).  Let $x\in \widetilde{S}_{\alpha }.$
Since $x\in \widetilde{F},$   we have $P\left(
x^{\prime }\right) =x$ for some $x^{\prime}\in F$,  and hence $f\left( x^{\prime }\right) =\left( 
\widetilde{f}\circ P\right) \left( x^{\prime }\right) =\widetilde{f}\left(
P\left( x^{\prime }\right) \right) =\widetilde{f}\left( x\right) \leq \alpha
,$ which proves that $x^{\prime }\in S_{\alpha }.$  Therefore $x\in P\left(
S_{\alpha }\right)$, which shows  $\widetilde{S}_{\alpha }\subset P\left(
S_{\alpha }\right).$ To prove the opposite inclusion, let $x\in P\left(
S_{\alpha }\right).$ We then have $x=P\left( x^{\prime }\right) $ for some $%
x^{\prime }\in S_{\alpha }.$ From the inclusion $S_{\alpha }\subset F,$ it
follows that $x\in P(F)=\widetilde{F}.$ On the other hand, $\widetilde{f}%
\left( x\right) =f\left( x^{\prime }\right) \leq \alpha .$  This shows
$x\in \widetilde{S}_{\alpha }$  and proves the inclusion $P\left(
S_{\alpha }\right) \subset \widetilde{S}_{\alpha }$ and hence our claim $%
\widetilde{S}_{\alpha }=P\left( S_{\alpha }\right).$

Altogether, $\widetilde{f%
}$ now attains its infimum on $\widetilde{F}$ by the induction hypothesis, 
and then $f$, having the same values, also attains
its infimum on $F$.
This proves the validity of (1).
\end{proof}

\begin{remark}
The equivalence of (4) and (6) can already be found in \cite{Kl60}.
\end{remark}

\begin{remark}
All that matters in condition (1) is the {\em rigidity} of polynomials. Any class $\mathcal F(L)$ of continuous functions
defined on affine subspaces $L$ of $\mathbb R^n$
with the following properties would work as well:
(i) $\mathcal F(L)$ is defined for every $L \subset \mathbb R^n$ and every  $n$.
(ii) If $f \in \mathcal  F(\mathbb R)$ is bounded below on a closed interval on $\mathbb R$, then $f$ attains its infimum. 
(iii) If $f\in \mathcal F(\mathbb R^n)$ and    $H$ is a hyperplane in $\mathbb R^n$, then
$f_{|H} \in \mathcal F(H)$. (iv) If $f \in \mathcal F(\mathbb R^n)$ is bounded (above and below)
on some ray $x + \mathbb R^+d \subset \mathbb R^n$, then $f$  does not depend on $d$, i.e.,
$f(x) = f(x+td)$ for all $t \in \mathbb R$. 
\end{remark}

We had seen in section \ref{sec:FWsets} that {\em FW}-sets have no
$f$-asymptotes. Moreover, from the results of this section we see that if $F$
is convex and has no $f$-asymptotes, then it is already a {\em qFW}-set. This rises the question whether the absence
of $f$-asymptotes also serves to characterize
{\em FW}-sets, or if not, whether it does so at least for convex $F$.  
We indicate by way of two examples that this is not the case, i.e.,  the absence of $f$-asymptotes does \emph{not}
characterize Frank-and-Wolfe sets. Or put  differently, there exist
quasi-Frank-and-Wolfe sets which are not Frank-and-Wolfe.

\begin{example}
\label{example_LZ}
We construct a closed convex set $F$ without $f$-asymptotes, which is not
Frank-and-Wolfe. We use Example 2 of \cite{LZ99}, which we reproduce here
for convenience. Consider the optimization program 
\begin{eqnarray*}
\begin{array}{ll}
\mbox{minimize} & q(x) = x_1^2 -2x_1x_2 + x_3x_4 +1 \\ 
\mbox{subject to} & c_1(x) = x_1^2 - x_3 \leq 0 \\ 
& c_2(x) = x_2^2-x_4 \leq 0 \\ 
& x\in \mathbb{R}^4%
\end{array}%
\end{eqnarray*}
then as Lou and Zhang \cite{LZ99} show the constraint set $F=\{x\in \mathbb{R%
}^4: c_1(x)\leq 0, c_2(x)\leq 0\}$ is closed convex, and the quadratic
function $q$ has infimum $\gamma=0$ on $F$, but this infimum is not
attained.

Let us show that $F$ has no $f$-asymptotes. Note that $F=F_1 \times F_2$,
where $F_1 = \{(x_1,x_3)\in \mathbb{R}^2: x_1^2-x_3\leq 0\}$, $F_2 =
\{(x_2,x_4)\in \mathbb{R}^2: x_2^2-x_4 \leq 0\}$. Observe that $F_1\cong F_2$%
, and that $F_1$ does not have asymptotes, being a parabola. Therefore, $F$
does not have $f$-asymptotes either. This can be seen from the following:
\end{example}



\begin{proposition}
\label{intersection}{\em
Any nonempty finite intersection of $qFW$-sets is again a $qFW$-set.}
\end{proposition}

\begin{proof}
By Theorem \ref{asymptote} the result follows immediately from
a theorem of Klee \cite[Thm. 4]{Kl60}, which says that
finite intersections of sets without $f$-asymptotes have no $f$-asymptotes.  
\end{proof}

\begin{corollary}
\label{product}{\em
If $F_1,\dots,F_m$ are $qFW$-sets, then the
cartesian product $F_1\times \dots \times F_m$ is again a $qFW$-set.}
\end{corollary}

\begin{proof}
Consider for the ease of notation the case of two sets $F_i \subset \mathbb R^{d_i}$, $i=1,2$. 
Then write
$$F_1 \times F_2 =\left( F_1 \times \mathbb R^{d_2}\right) \cap
\left(\mathbb R^{d_1} \times F_2 \right).$$ 
Now $F_1 \times \mathbb R^{d_2}$ is also {\em qFW}, and so
is $\mathbb R^{d_1} \times F_2$, and hence
the result follows from Proposition
\ref{intersection}. The fact that $F_1 \times \mathbb R^{d_1}$ is {\em qFW} is easily seen as follows: If $M$ is a $f$-asymptote of
$F_1 \times \mathbb R^{d_1}$, then $L = \{x: (x,y) \in M \mbox{ for some }y\}$ is a $f$-asymptote of $F_1$.
\end{proof}

\begin{remark}
Example 1 also tells us that the sum of $FW$-sets need not be a $FW$-set even when closed, as follows from the identity 
$F_1 \times F_2 = (F_1 \times \{0\}) + (\{0 \} \times F_2)$. Note that even though $F_1 \times F_2$ fails to be $FW$, it remains $qFW$ due to Corollary \ref{product}.
\end{remark}

\begin{example}
Let $F$ be the epigraph of $f(x)=x^2 + \exp(-x^2)$ in $\mathbb{R}^2$. Then $%
q(x,y)= y-x^2$ is bounded below on $F$, but does not attain its infimum, so $F
$ is not \emph{FW}. However, $F$ has no $f$-asymptotes, so it is \emph{qFW}.
\hfill $\square$
\end{example}

\begin{remark}
In \cite{JCA} it is shown explicitly that the ice-cream cone is not {\em qFW}. Here is a simple synthetic argument.
The ice cream cone $D \subset \mathbb R^3$ can be cut by a plane $L$ in such a way that $F = D \cap L$ has a hyperbola as boundary
curve. Since $F$ has asymptotes, it is not {\em qFW}, hence neither is the cone $D$.
\end{remark}

The method of proof in implication (6) $\implies$ (1) in Theorem \ref{asymptote}  can be used to show that
sub-level sets of convex polynomials are {\em qFW}-sets, see \cite[Chap. II, $\mathsection 4$, Thm. 13]{belusov}. We obtain the following
extension of \cite[Thm. 3]{klatte}:

\begin{corollary}
\label{cor4}
{\em
Let $F_0$ be a $qFW$-set and let
$f_1,\dots,f_m$ be convex polynomials on $F_0$ such that  the set $F=\{x\in F_0: f_i(x) \leq 0, i=1,\dots,m\}$
is non-empty. 
Let
$f$ be a polynomial which is bounded below on $F$ and has at least one nonempty convex sub-level set on $F$.
Then $f$ attains its infimum on $F$.}
\hfill $\square$
\end{corollary}

\begin{remark}
From Corollary \ref{product} and Proposition \ref{intersection} we learn that the class of {\em qFW}-sets is closed under
finite intersections and cross products, while example \ref{example_LZ} tells us that this is no longer true for {\em FW}-sets.
Yet another invariance property of {\em qFW}-sets is the following:
\end{remark}

\begin{corollary}{\em
Let $T:\mathbb R^n \to \mathbb R^m$ be an affine operator, and let $F \subset \mathbb R^m$ be a qFW-set. If
$T^{-1}(F)$ is nonempty, then it is a qFW-set, too.}
\end{corollary}

\begin{proof}
We use property (4) of Theorem \ref{asymptote}.
Suppose  $T^{-1}(F)$ had an $f$-asymptote $M$, then $T(M)$ would be an $f$-asymptote of $F$.
\end{proof}


\begin{corollary}
{\rm (See \cite{klatte}, \cite[Cor. 27.3.1]{rock}).}  {\em
Let $f$ be a polynomial which  is convex and  bounded below on
a {qFW}-set $F$. Then $f$ attains its infimum on $F$.}
\hfill $\square$
\end{corollary}

The following consequence of Theorem \ref{asymptote}
is surprising.

\begin{corollary}
\label{surprise}
{\em
Let $F$ be a convex cone. Then the following are
equivalent:
\begin{enumerate}
\item[\rm (1)]
 $F$ is a FW-set;
 \item[\rm (2)] $F$ is a qFW-set;
 \item[\rm (3)]  $F$  is polyhedral.
 \end{enumerate}
 }
\end{corollary}

\begin{proof}
(1) $\implies$ (2) is clear, because $F$ is convex.
(2) $\implies$ (3):
Let $F\subset \mathbb R^n$ be \emph{qFW}, then by condition (iv) of Theorem \ref{asymptote} 
every orthogonal projection
$P(F)$ on any two-dimensional subspace of $\mathbb R^n$ is closed.  
Therefore, by Mirkil's theorem, which we give as Lemma \ref{mirkil} below,
$F$ is polyhedral. 
(3) $\implies$ (1):
By the classical Frank-and-Wolfe theorem every polyhedral convex cone is \emph{FW}. 
\end{proof}

\begin{lemma}
\label{mirkil}
{(Mirkil's theorem \cite{mirkil})}. {\em Let $D$ be a convex cone in $\mathbb R^n$
such that every orthogonal projection on any two-dimensional subspace is closed.
Then $D$ is polyhedral.}
\hfill $\square$
\end{lemma}

\begin{remark}
This result puts an end to hopes to get new results for the linear complementarity problem
by investigating $FW$-cones. 
\end{remark}

We end this section with a nice consequence of Mirkil's theorem.
First we need the following characterization of
$f$-asymptotes:

\begin{proposition}
\label{f-asympt}  {\em For a closed convex set $F$ and a linear subspace $L,$ the
following statements are equivalent:

{\rm 1)} No translate of $L$ is an f-asymptote of $F.$

{\rm 2)} The orthogonal projection of $F$ onto the orthogonal complement $L^\perp$ is
closed.

{\rm 3)}   $F+L$ is closed.
}
\end{proposition}

\begin{proof}
1) $\Rightarrow $ 2). Let $x\in \overline{P_{L^{\perp }}\left( F\right)}.$ Since $%
P_{L^{\perp }}^{-1}\left( x\right) =x+L,\ $we can easily prove that 
dist$\left( F,x+L\right) ={\rm dist}\left( P_{L^{\perp }}\left( F\right) ,x\right)
=0$.  Since $x+L$ is not an $f$-asymptote of $F,$ we have $F\cap \left(
x+L\right) \neq \emptyset ,$ which amounts to saying that $x\in P_{L^{\perp
}}\left( F\right) .$

2) $\Rightarrow $ 3). Let $x_{k}\in F$ and $y_{k}\in L$ $\left(
k=1,2,...\right) $ be such that the sequence $x_{k}+y_{k}$ converges to some
point $z.$ Then $P_{L^\perp}(z)=\lim P_{L^{\perp }}\left( x_{k}+y_{k}\right) =\lim P_{L^{\perp
}}\left( x_{k}\right) \in P_{L^{\perp }}\left( F\right)$ 
 due to closedness of $P_{L^\perp}(F)$.  But $P_{L^\perp}(F) =\left( F+L\right)
\cap L^{\perp }\subset F+L$, hence $P_{L^\perp}(z) \in F+L$. Now 
 $z = P_{L^\perp}(z) + P_L(z) \in F +L+L = F+L$.

3) $\Rightarrow $ 1). Let as assume that $x+L$ is an $f$-asymptote of $F$ for
some $x.$ Then $0\leq {\rm dist}\left( x,F+L\right) \leq {\rm dist}\left( x,\left(
F+L\right) \cap L^{\perp }\right) ={\rm dist}\left( x,P_{L^{\perp }}\left(
F\right) \right) ={\rm dist}\left( F,x+L\right) =0,$ hence ${\rm dist}\left(
x,F+L\right) =0$.   Since $F+L$ is closed, this implies $x\in F+L.$ This
is equivalent to saying that $F\cap \left( x+L\right) \neq \emptyset ,$ a
contradiction to the assumption that $x+L$ is an $f$-asymptote of $F.$
\end{proof}

The  consequence of Mirkil's Theorem we have in mind is the following:

\begin{proposition}
{\em
For a closed convex cone $D$ in $\mathbb R^{n}$ (with $n>2$),  the following statements
are equivalent:

{\rm 1)}  $D$ is polyhedral.

{\rm 2)}  $C+D$ is a convex polyhedron for every convex polyhedron $C$.

{\rm 3)}  $L+D$  is closed for every $(n-2)$-dimensional subspace $L$.

{\rm 4)}  $D$ has no $(n-2)$-dimensional f-asymptotes.
}
\end{proposition}

\begin{proof}
Implications 1) $\Rightarrow $ 2) $\Rightarrow $ 3) are immediate.
Implication 3) $\Longrightarrow $ 1) is a consequence of 3) $\Rightarrow $
2) of Proposition \ref{f-asympt} combined with Mirkil's Theorem. Implication
3) $\Longrightarrow $ 4) follows from 3) $\Longrightarrow $ 1) of
Proposition \ref{f-asympt}. Finally, implication 4) $\Longrightarrow $ 3)
can be easily derived from implication 1) $\Longrightarrow $ 3) of
Proposition \ref{f-asympt}.
\end{proof}

\section{Motzkin type sets}
\label{sec:Motzkin}

Following \cite{goberna,IMT14}, 
a convex set $F$ is called Motzkin decomposable,
if it may be written as the Minkowski sum of a compact convex set $C$
and a closed convex cone $D$, that is, $F = C + D$.  Motzkin's classical result states that every
convex polyhedron 
has such a decomposition. We extend this
definition as follows:

\begin{definition}
A  closed set $F\subset \mathbb R^n$ is called a {\em Motzkin set}, for short an {\em M}-set,  if it can be written as $F = K + D$, where
$K$ is a compact set, and $D$ is a closed convex cone.
\end{definition}

We shall continue to reserve the term Motzkin decomposable for the case where
the set $F$ is convex. A Motzkin set $F$ which is convex is then clearly Motzkin decomposable.

\begin{remark}
Let $F = K + D$ be a Motzkin set, then similarly to the convex case
$D$ is uniquely determined by $F$. Indeed, taking convex hulls, we have co$(F) = {\rm co}(K)+{\rm co}(D) = {\rm co}(K)+D$, hence
${\rm co}(F)$ is a convex Motzkin set, i.e., a Motzkin decomposable set. Then from known results
on Motzkin decomposable sets \cite{goberna,IMT14}, $D = 0^+{\rm co}(F)$, the recession cone of co$(F)$. 
Now if we  define the recession cone of $F$ in the same way as in the convex case, i.e., 
$0^+F = \{u \in \mathbb R^n: x + tu\in F \mbox{ for all $x\in F$ and all $t\geq 0$}\}$, then
$0^+F \subset 0^+{\rm co}(F) = D \subset 0^+F$, proving $D = 0^+F$. 
In particular, $F$ and co$(F)$ have the same recession cone.
\end{remark}

\begin{theorem}
\label{kummer}
{\em
Let $F$ be a Motzkin set in $\mathbb R^n$, represented as  $F = K+D = K + 0^+F$.
Then the following are equivalent:
\begin{enumerate}
\item[\rm (1)] $F$ is a $FW$-set.
\item[\rm (2)] The recession cone $0^+F$ of $F$ is polyhedral.
\item[\rm (3)] $F$ has no $f$-asymptotes.
\end{enumerate}
}
\end{theorem}

\begin{proof}
We prove (1) $\implies$ (2).
Let $P$ be an orthogonal projection of 
$\mathbb R^n$ onto a  subspace $L$ of $\mathbb R^n$.
Since $F = K+ D$ is a $FW$-set, $P(F)$ is closed. 
Since $P(F) = P(K) + P(D)$ and $\overline{P(F)} = P(K) + \overline{P(D)}$, this
means
$P(K) + P(D) = P(K)+ \overline{P(D)}$. 
We have to show that this implies $P(D)=\overline{P(D)}$.
This follows from the so-called {\em order cancellation law}, which we give 
as Lemma \ref{cancellation} below. It is
applied to the convex sets $A=\overline{P(D)}$, $B = P(D)$, and for the compact set $P(K)$.
This shows indeed $\overline{P(D)} = P(D)$. This means every
projection of $D$ is closed, hence by Mirkil's theorem (Lemma \ref{mirkil}), the cone
$D$ is polyhedral.

\begin{lemma}
\label{cancellation}
{\rm (Order cancellation law, see \cite{IMT14})}. {\em Let $A,B\subset \mathbb R^n$ be convex sets,
$K\subset \mathbb R^n$ a compact set. If $A+K \subset B+K$, then $A \subset B$.}
\hfill $\square$
\end{lemma}

Let us now prove (2) $\implies$ (1). 
Write $F = K+ D$ for $K$ compact and $D$ a polyhedral convex cone. Now consider a quadratic
function $q(x)=\frac{1}{2}x^\mathsf{T} Ax + b^\mathsf{T} x$ bounded below by 
$\gamma$ on $F$. Hence 
\begin{equation}  \label{inf_inf}
\inf_{x\in F} q(x) = \inf_{y\in K} \inf_{z \in D} q(y+z) =\inf_{y\in K}
\left( q(y) + \inf_{z \in D} y^\mathsf{T} Az + q(z)\right) \geq \gamma.
\end{equation}
Observe that for fixed $y\in K$ the function $q_y:z \mapsto y^\mathsf{T}
Az + q(z)$ is bounded below on $D$ by $\eta = \gamma - \max_{y^{\prime
}\in C} q(y^{\prime })$. Indeed, for $z\in D$ we have 
\begin{align*}
y^\mathsf{T} Az + q(z) &\geq \left( q(y) + \inf_{z^{\prime}\in  D} y^\mathsf{T}
Az^{\prime }+ q(z^{\prime })\right) - q(y) \\
&\geq \inf_{y\in K} \left( q(y) + \inf_{z^{\prime}\in D} y^\mathsf{T}
Az^{\prime }+q(z^{\prime })\right) - \max_{y^{\prime }\in K} q(y^{\prime })
\\
& \geq \gamma - \max_{y^{\prime }\in K} q(y^{\prime })=\eta.
\end{align*}
Since $q_y$ is a quadratic function bounded below on the polyhedral cone $%
D$, the inner infimum is attained at some $z = z(y)$. This is in fact the
classical Frank and Wolfe theorem on a polyhedral cone. In consequence the
function $f:\mathbb{R}^n \to \mathbb{R }\cup\{-\infty\}$ defined as 
\begin{equation*}
f(y) = \inf_{z\in D} y^\mathsf{T} Az +q(z), 
\end{equation*}
satisfies $f(y) = y^\mathsf{T} Az(y) + q(z(y))>-\infty$ for every $y\in K$, so the
compact set $K$ is contained in the domain of $f$. But now a stronger result
holds, which one could call a parametric Frank and Wolfe theorem, and which
we shall prove in Lemma \ref{FW2} below. We show that $f$ is continuous relative to
its domain. Once this is proved, the infimum (\ref{inf_inf}) can then be
written as 
\begin{equation*}
\inf_{x\in F} q(x) = \inf_{y\in K} q(y) + f(y), 
\end{equation*}
and this is now attained by the Weierstrass extreme value theorem due to the
continuity of $q+f$ on the compact $K$. Continuity of $f$ on $K$ is now a consequence of the following
\end{proof}

\begin{lemma}
\label{FW2} {\em Let $D$ be a polyhedral convex cone and define 
\begin{equation*}
f(c)=\inf_{x\in D}c^{\mathsf{T}}x+\textstyle\frac{1}{2}x^{\mathsf{T}}Gx, 
\end{equation*}%
where $G=G^{\mathsf{T}}$. Then $\mathrm{dom}(f)$ is a polyhedral convex
cone, and $f$ is continuous relative to $\mathrm{dom}(f)$.}
\end{lemma}

\begin{proof}
If $x^{\mathsf{T}}Gx<0$ for some $x\in D,$ then $\mathrm{dom}(f)=\emptyset ,$
so we may assume for the remainder of the proof  that $x^{\mathsf{T}}Gx\geq 0$ for every $x\in D.$   
The proof is now divided into three parts. In part 1) we establish a formula for the domain
dom$(f)$. In part 2) we use this formula to show that dom$(f)$ is polyhedral, and in part 3) we show that the latter implies
continuity of $f$ relative to dom$(f)$. 
\\

1) We start by proving that %
\begin{equation}
\label{domain}
\mathrm{dom}(f)=\left\{ c:c^{\mathsf{T}}x\geq 0\text{ for every }x\in D\text{
such that }x^{\mathsf{T}}Gx=0\right\} . 
\end{equation}%
The inclusion $\subseteq $ being obvious,  we  have to prove  the
following implication:%
\begin{equation*}
c^{\mathsf{T}}x\geq 0\text{ for every }x\in D\text{ such that }x^{\mathsf{T}%
}Gx=0\Longrightarrow \inf_{x\in D}c^{\mathsf{T}}x+\textstyle\frac{1}{2}x^{%
\mathsf{T}}Gx>-\infty .
\end{equation*}%
We establish this by induction on the number $l$ of generators of $D.$ The
case $l=1$ being clear, let  $l>1$, and suppose the  implication is correct for every polyhedral convex cone $D'$ 
with $l' < l$ generators.
Let $c$ be such that $c^{\mathsf{T}%
}x\geq 0$ for every $x\in D$ having $x^{\mathsf{T}}Gx=0.$  
We have to show that $c\in {\rm dom}(f)$. 
Assume on the contrary that
\begin{equation}
\inf_{x\in D}c^{\mathsf{T}}x+\textstyle\frac{1}{2}x^{\mathsf{T}}Gx=-\infty,
\label{unbounded}
\end{equation}
and choose a sequence $x_{k}\in D$ with $\|x_k\|\to \infty$ such that%
\begin{equation}
c^{\mathsf{T}}x_{k}+\textstyle\frac{1}{2}x_{k}^{\mathsf{T}%
}Gx_{k}\longrightarrow -\infty .  \label{seq -infty}
\end{equation}%
Passing to a subsequence, we can assume that the sequence $y_{k}={x_{k}}/{\left\Vert x_{k}\right\Vert }$
converges to some $y\in D.$ We must have $y^{\mathsf{T}}Gy=0,$ as
otherwise we would have $c^{\mathsf{T}}x_{k}+\textstyle\frac{1}{2}x_{k}^{%
\mathsf{T}}Gx_{k}=\left\Vert x_{k}\right\Vert c^{\mathsf{T}}y_{k}+\textstyle%
\frac{1}{2}\left\Vert x_{k}\right\Vert ^{2}y_{k}^{\mathsf{T}%
}Gy_{k}\longrightarrow +\infty ,$ a contradiction. 
Hence, by our assumption, 
$c^{\mathsf{T}}y\geq 0.$ We cannot have $c^{\mathsf{T}}y>0,$  as otherwise
for large enough $k$ we would have $c^{\mathsf{T}}x_{k}=\left\Vert
x_{k}\right\Vert c^{\mathsf{T}}y_{k}>0$ and thus $c^{\mathsf{T}}x_{k}+%
\textstyle\frac{1}{2}x_{k}^{\mathsf{T}}Gx_{k}>0$ due to $x_k^{\mathsf{T}}Gx_k \geq 0$, which is impossible
because of (\ref{seq -infty}). Therefore $c^{\mathsf{T}}y=0.$  This will be used later.

Collecting more facts about $y$, 
note that as a consequence of our standing assumption $x^{\mathsf{T}}Gx \geq 0$ for $x\in D$, $y$ is a minimizer of the quadratic form $\textstyle\frac{1}{2}x^{%
\mathsf{T}}Gx$ over $D,$ which implies that $Gy$ belongs to the positive
polar cone of $D,$ that is, $x^{\mathsf{T}}Gy\geq 0$ for every $x\in D.$ 
This property will also be used below.

Let 
$E=\left\{ e_{1},...,e_{l}\right\} $ be the set of generating rays of $D,$ and
for $i=1,...,l$ denote by $D_{i}$ and $\widehat{D}_{i}$ the cones generated     
by $E\setminus \left\{ e_{i}\right\} $ and $\left( E\setminus \left\{
e_{i}\right\} \right) \cup \left\{ y\right\} ,$ respectively. As the
induction hypothesis applies to each $D_i$, we have $\inf_{x\in D_{i}}c^{\mathsf{T}}x+%
\textstyle\frac{1}{2}x^{\mathsf{T}}Gx>-\infty$ for every $i$,  so the infimum $m$ of $c^{%
\mathsf{T}}x+\textstyle\frac{1}{2}x^{\mathsf{T}}Gx$ over $%
\bigcup\limits_{i=1}^{l}D_{i}$ is finite.

Now observe that
\begin{equation}
\label{newclaim}
D=\bigcup\limits_{i=1}^{l}\widehat{D}_{i}.
\end{equation}
Indeed, the inclusion $\supseteq$ being clear, take $x\in D$ and write it as $x=\sum_{i=1}^l \lambda_ie^i$ for certain $\lambda_i \geq 0$.
Since $y\in D \setminus\{0\}$, we have $y =\sum_{i\in I} \mu_ie^i$
for some $\emptyset \not= I \subset \{1,\dots,l\}$ and $\mu_i > 0$.  Put $\nu = \min\{\lambda_i/\mu_i: i \in I\} =: \lambda_{i_0}/\mu_{i_0}$, then
\begin{align*}
x &= \sum_{i\in I} \lambda_ie^i + \sum_{j\not\in I} \lambda_je^j + \nu \left( y- \sum_{i\in I} \mu_i e^i \right) 
= \sum_{i\in I} \left(\lambda_i - \nu \mu_i   \right) e^i + \sum_{j\not\in I} \lambda_j e^j + \nu y.
\end{align*}
Since $\lambda_i - \nu \mu_i \geq 0$ for every $i\in I$, and $\lambda_{i_0} - \nu \mu_{i_0}=0$, we have
shown $x \in \widehat{D}_{i_0}$. That proves (\ref{newclaim}).

Now, using (\ref{newclaim}),  for every $%
x\in D$ there exist $i\in \left\{ 1,...,l\right\} ,$ $z\in D_{i},$ and $%
\lambda \geq 0$  such that $x=z+\lambda y.$ We then have $c^{\mathsf{T}}x+ %
\textstyle\frac{1}{2}x^{\mathsf{T}}Gx=c^{\mathsf{T}}z+ \lambda c^{\mathsf{T}} y + \textstyle\frac{1}{2}%
z^{\mathsf{T}}Gz+\lambda z^{\mathsf{T}}Gy + \frac{1}{2} \lambda^2 y^{\mathsf{T}} Gy
= c^{\mathsf{T}} z + \frac{1}{2} z^{\mathsf{T}} Gz + \lambda z^{\mathsf{T}} Gy \geq c^{\mathsf{T}}z+ \frac{1}{2}z^{\mathsf{T}}Gz \geq  m,$ which gives $%
\inf_{x\in D}c^{\mathsf{T}}x+\textstyle\frac{1}{2}x^{\mathsf{T}}Gx=m,$ 
contradicting (\ref{unbounded}). This shows that our claim (\ref{domain}) was correct.
\\

2)
Now by the Farkas-Minkowski-Weyl theorem (cf. \cite[Thm. 19.1]{rock} or \cite[Cor. 7.1a]{schrijver}) the polyhedral cone
$D$ is the linear image of the positive orthant of a space $\mathbb R^p$ of appropriate dimension,
i.e. $D=\{Zu: u\in \mathbb R^p, u\geq 0\}$. Using (\ref{domain}), this implies
\[
{\rm dom}(f)=\{c: c^\ttop Zu \geq 0 \mbox{ for every } u \geq 0 \mbox{ such that } u^\ttop Z^\ttop GZu=0\}.
\]
Now observe that  if $u\geq 0$ satisfies $u^\ttop Z^\ttop GZu=0$, then it is a minimizer
of the quadratic function $u^\ttop Z^\ttop GZu$ on the cone $u\geq 0$, hence $Z^\ttop GZu\geq 0$ by the Kuhn-Tucker conditions. 
Therefore we can write the set $P=\{u\in \mathbb R^p: u\geq 0, u^\ttop Z^\ttop GZu=0\}$ as
\[
P = \bigcup_{I\subset \{1,\dots,p\}} P_I  ,
\]
where the $P_I$ are the polyhedral convex cones
$$
P_I= \{u\geq 0: Z^\ttop GZu \geq 0, u_i=0 \mbox{ for all }i\in I, (Z^\ttop GZu)_j=0 \mbox{ for all }j\not\in I\}.
$$
For every $I\subset \{1,\dots,p\}$ choose $m_I$  generators $u_{I1},\dots,u_{Im_I}$ of $P_I$. 
Then,  
\begin{align}
\label{dom}
{\rm dom}(f) &= \left\{c: c^\ttop Zu \geq 0 \mbox{ for every } u\in P\right\} \\
&=\left\{c: c^\ttop Zu \geq 0 \mbox{ for every } u\in \textstyle \bigcup_{I\subset\{1,\dots,p\}} P_I\right\} \notag \\
&=\textstyle \bigcap_{I\subset \{1,\dots,p\}} \left\{c: c^\ttop Zu \geq 0 \mbox{ for every } u \in P_I\right\}  \notag\\
&= \textstyle \bigcap_{I\subset \{1,\dots,p\}} \left\{c: c^\ttop Zu_{Ij} \geq 0 \mbox{ for all } j=1,\dots,m_I\right\}. \notag
\end{align}
Since a finite intersection of polyhedral cones is polyhedral, 
this proves that dom$(f)$ is a polyhedral convex cone. 

3)
To conclude, continuity of
$f$ relative to its domain now follows from  polyhedrality of dom$(f)$, and using \cite[Thm. 10.2]{rock}, since $f$ is clearly concave and
upper semicontinuous. This completes the proof of (2) $\implies$ (1).

(1) $\implies$ (3) was proved in Proposition \ref{FWnof}. Let us prove (3) $\implies$ (2). By Mirkil's theorem (Lemma \ref{mirkil}) it suffices to show that every orthogonal projection
$P(F)$ is closed. 
Suppose this is not the case, and let $y \in \overline{P(F)}\setminus P(F)$. Let $L = y + {\rm ker}(P)$, then
$F \cap L = \emptyset$. Now choose $y_n \in F$ such that $P(y_n) \to P(y)=y$. Then
$y_n = P(y_n) + z_n$ with $z_n \in {\rm ker}(P)$. Hence $P(y) + z_n\in L$, but
$\|(P(y_n) + z_n) - (P(y)+z_n) \| \to 0$, which shows dist$(F,L)=0$. That means $F$ has an $f$-asymptote, a contradiction. 
\end{proof}

\begin{remark}
The main implication (2) $\implies$ (1)  in Theorem \ref{kummer}  
was first proved by Kummer \cite{kummer}.
Our proof of (2) $\implies$ (1) is  slightly stronger in so far as it gives additional information
on the polyhedrality of the  domain of $f$ in Lemma \ref{FW2}.
\end{remark}


\begin{remark}
We refer to Bank \emph{et al.} \cite[Thm. 5.5.1 (4)]{banks} 
for a  result related to Lemma \ref{FW2} in the case where $G\succeq 0$. For
the indefinite case see also Tam \cite{tam}.
\end{remark}

\begin{remark}
\label{rem3} The statement of Theorem \ref{kummer} is no longer correct if one drops the hypothesis
that $F$ is a Motzkin set. We take the convex $F=\{(x,y)\in \mathbb{R}^2: x>0,
y>0,    \,xy \geq 1\}$, then $F$, being limited by a hyperbola, has $f$-asymptotes, hence is not \emph{qFW}, but $0^+F$
is the positive orthant, which is polyhedral.
\end{remark}

\begin{corollary}
\label{resume}
{\em
A Motzkin decomposable set $F$ without $f$-asymptotes is Frank-and-Wolfe. }
\end{corollary}

\begin{proof}
Since $F$ has no $f$-asymptotes and is convex, it is a $qFW$-set by Theorem \ref{asymptote}. But  then
by Theorem \ref{kummer}, $F$ is even a $FW$-set.
\end{proof}

\section{Invariance properties of Motzkin {\em FW}-sets}
\label{invariance}
We have seen in example \ref{example_LZ} that intersections of $FW$-sets
need no longer be $FW$-sets, not even when convexity is assumed. In contrast, the class of
$qFW$-sets turned out closed under finite intersections. This rises the question whether more
amenable sub-classes of the class of {\em FW}-sets with better invariance properties may be
identified. In response we show
in this chapter that the class of Motzkin {\em FW}-sets, for short {\em FWM}-sets, is better behaved with
regard to invariance properties.

\begin{lemma}
\label{motzkin_inter}
{\em
Consider a set of the form $K +D$, where $K$ is compact and $D$ is a polyhedral closed convex cone in $\mathbb R^n$.
Let $L$ be a linear subspace of $\mathbb R^n$. Then there exists a compact set $K_0$
such that $(K+D) \cap L = K_0 + (D\cap L)$.
}
\end{lemma}

\begin{proof}
1)
We assume 
for the time being that the cone $D \cap L$ is pointed. 
For fixed $x\in K$ consider the polyhedron $P_x := (x+D) \cap L$. Define
$M(P_x) = \{x' \in P_x : (x'-(D\cap L)) \cap P_x = \{x'\}\}$, and let $K(P_x)$ be the closed convex hull of
$M(P_x)$. Then according to \cite[Thm. 19]{goberna} the set $K(P_x)$ is compact, and  we have the minimal Motzkin decomposition $P_x = K(P_x) + (D\cap L)$.
This uses the fact that $D \cap L$ is the recession cone of  $P_x$.  It follows that
\[
(K + D) \cap L = \bigcup_{x\in K} (x+D) \cap L = \bigcup_{x\in K} K(P_x) + (D \cap L),
\]
so all we have to do is show that the set $\cup_{x\in K} K(P_x)$ is bounded, as then its closure $K_0$ is the compact set
announced in the statement of the  Lemma. To prove boundedness of   $\cup_{x\in K} K(P_x)$ it clearly suffices to show that
$\cup_{x\in K} M(P_x)$ is bounded.


Let $\mathcal F$ be the finite set of faces of $D$, where we assume that $D$ itself is 
a face. Let ${x}' \in M(P_x)$, then $x'$ is in the relative interior of one of the faces  $x + F$, $F\in \mathcal F$,  of the shifted cone $x+D$. 

We divide the faces $F \in \mathcal F$ of the cone $D$ into two types: 
$\mathcal F_1$ is the class of those faces $F\in \mathcal F$ for which there exists $d \in L$, $d\not=0$,  
such that $d$ is a direction of recession of $F$, i.e., those where $F \cap L$ does not reduce to $\{0\}$. 
The class $\mathcal F_2$ gathers the remaining faces of $D$  which are not in the class $\mathcal F_1$.

Now suppose the set $\bigcup_{x\in K} M(P_x)$ is unbounded. Then there exists
a sequence $x_k\in K$ and ${x}_k' \in M(P_{x_k})$ with $\|{x}_k'\|\to\infty$. 
From the above we know that each ${x}_k'$ is in the relative interior of $x_k + F_k$ for some $F_k\in \mathcal F$.
Since there are only
finitely many faces,
we can extract a subsequence, also denoted $x_k$ and satisfying $\|{x}_k'\|\to\infty$,   
such that the ${x}_k'$ are relative interior points of $x_k + F$ for the same fixed face $F\in \mathcal F$. Due to compactness of $K$
we may, in addition, assume that $x_k \to x\in K$. Using the definition of $M(P_{x_k})$
write ${x}_k' = x_k + t_k d_k\in L$ with $d_k\in F\subset D$, $\|d_k\|=1$, $t_k > 0$, $t_k \to\infty$. 
Passing to yet another subsequence, assume that $d_k\to d$, where $\|d\|=1$.
It follows that $d\in L$, because in the expression
$
{{x}_k'}/{t_k} = {x_k}/{t_k} + d_k
$
the middle term tends to 0 due to compactness of $K$ and $t_k \to \infty$, while the left hand term is in $L$ because $x_k'$ belongs to $L$. 
Since $F$ is a cone, it also follows that $x + \mathbb R_+d \subset x + F$, hence
$d\in F$.
This shows that the face $F$ is in the class $\mathcal F_1$.

2)
So far we
have shown that
$\bigcup_{F \in \mathcal F_2} \{x' \in M(P_x):  x\in K, x' \in {\rm ri}(x+F) \}$ is a bounded set. 
It
 remains to prove that this set contains already all points $x' \in M(P_x)$,  $x\in K$, i.e., that
$x' \in M(P_x)$ cannot be a relative interior point of any of the faces $x+F$ with $F \in \mathcal F_1$.

3)
Contrary to what is claimed, 
consider $x\in K\setminus L$ such that $x'\in M(P_x)$ satisfies ${x}'\in {\rm ri}(x+F)$ for some  $F \in \mathcal F_1$. By definition of the class $\mathcal F_1$
there exists $d \in L \cap F$, $d\not=0$. 
Since ${x}' \in L$ by the definition of $M(P_x)$, 
we have $x' + \mathbb R d \subset L$.  But this line is also contained in $x+{\rm span}(F)$, because
we have $d\in {\rm span}(F)$ and $x' = x + d'$ for some $d' \in F$, hence 
$x' + \mathbb R d \subset x+ {\rm span}(F)$.

Since ${x}'$ is a relative interior point of $x+F$, there exists $\epsilon > 0$
such that $N_\epsilon =\{{x}' + sd: |s| < \epsilon\}$ is contained in $x+F$.  Since $d\in F\cap L\subset D\cap L$, we have arrived at a contradiction with
the fact that ${x}' \in M(P_x)$. Namely, moving in $N_\epsilon$ we can stay in $P_x$ while going from
$x'$ slightly in the direction of $-d\in -(D\cap L)$. This contradiction shows that what was claimed in 2) is true.
The Lemma is therefore proved for pointed $D \cap L$.


4)
Suppose now $D$ is allowed to contain lines.   With a change of coordinates we may arrange that
$\mathbb R^n = \mathbb R^m \times \mathbb R^p$ and $D \subset \mathbb R^m \times \{0\}$, where the 
possibility $p=0$ is not excluded and corresponds to the case where $D-D=\mathbb R^n$. 
Now consider the space $\mathbb R^m \times \mathbb R^m \times \mathbb R^p$ and define the cone $\widetilde{D} \subset \mathbb R^m \times \mathbb R^m \times \mathbb R^p$ as
$\widetilde{D} = \{(x^+,x^-,0): x^\pm \in \mathbb R^m, x^\pm \geq  0, x^+-x^- \in D \}$. Then $\widetilde{D}$
is polyhedral  and pointed. Let $T$ be the mapping $(x^+,x^-,y) \mapsto (x^+-x^-,y)$, then  $T(\widetilde{D})=D$. Since
$T$ maps $\mathbb R^m \times \mathbb R^m \times \mathbb R^p$  onto $\mathbb R^m \times \mathbb R^p$, there exists a compact set $\widetilde{K}\subset \mathbb R^m\times \mathbb R^m \times \mathbb R^p$
such that $T(\widetilde{K}) = K$. Put $\widetilde{L} = T^{-1}(L)$. Now since $\widetilde{D}$ is pointed, the first part of the proof
gives a compact $\widetilde{K}_0 \subset \mathbb R^m \times \mathbb R^m \times \mathbb R^p$ 
such that $(\widetilde{K}+\widetilde{D}) \cap \widetilde{L} = \widetilde{K}_0 + (\widetilde{D} \cap \widetilde{L})$. Applying $T$
on both sides, and using the fact that
$\widetilde{L}$ is a pre-image, we deduce $(K+D) \cap L=T(\widetilde{K}_0) + (D \cap L)$. On putting 
$K_0 = T(\widetilde{K}_0)$ which is compact, we get the desired statement
$(K+D)\cap L=K_0 + (D \cap L)$.  That completes the proof
of the Lemma.
\end{proof}

\begin{corollary}
\label{FWM}
{\em
Any finite intersection of sets of the form $K + D$ with $K$ compact and 
$D$ a polyhedral convex cone is again a set of this form.}
\end{corollary}

\begin{proof}
It suffices to consider the case of two sets $F_i = K_i + D_i$ in $\mathbb R^n$, $i=1,2$,  with compact $K_i$ and $D_i$ polyhedral convex cones.
We build the set $F = F_1 \times F_2$ in $\mathbb R^n \times \mathbb R^n$, which is of the same form,  
because trivially $(K_1 + D_1) \times (K_2 + D_2) = (K_1 \times K_2) + (D_1 \times D_2)$, and since
the product of two polyhedral cones is a polyhedral cone.

Now by Lemma \ref{motzkin_inter}  the intersection of $F_1 \times F_2$
with the diagonal $\Delta = \{(x,x): x \in \mathbb R^n\}$ is a set of the form $\mathcal K + \mathcal D$ with $\mathcal K$ compact and
$\mathcal D$ a polyhedral convex cone, because
the diagonal is a linear subspace. Finally,  $F_1 \cap F_2$ is the image
of $\mathcal K + \mathcal D$ under the projection $p:(x,y)\to x$ onto the first coordinate, hence  is 
of the form $p(\mathcal K) + p(\mathcal D)$, and since $p(\mathcal D)$ is a polyhedral convex cone, we are done.
\end{proof}

We conclude with the following invariance property of the class $FWM$:

\begin{proposition}
\label{pre-image}
{\em
If the pre-image of a FWM-set under an affine mapping is nonempty, then it is a FWM-set.}
\end{proposition}

\begin{proof}
Let $T$ be an affine mapping and $F$ be a $FWM$-set such that $T^{-1}\left(
F\right) \neq \emptyset $. Since translates of $FWM$-sets are $FWM$, there is no
loss of generality in assuming that $T$ is linear. Then the restriction of $T$ to ker$(T)^\perp$ is a bijection
from ${\rm ker}(T)^{\perp}$ onto $R\left(T\right)$, and one has%
\[
T^{-1}\left(F\right) =\left( T_{|{\rm ker}(T)^{\perp}}\right)
^{-1}\left(F\cap R\left(T\right)\right) +  {\rm ker}(T).
\]%
Since $R\left(T\right) $ is a subspace, hence a convex polyhedron, and $%
T^{-1}\left(F\right) \neq \emptyset$,  the set $F\cap
R\left( T\right) $ is $FWM$ by Corollary \ref{FWM}. Since 
$\left( T_{|{\rm ker}(T)^{\perp}}\right)
^{-1}$
is an isomorphism from $R\left( T\right) $ onto 
${\rm ker}(T)^{\bot}$,  the set 
$
\left( T_{|{\rm ker}(T)^{\perp}}\right)
^{-1}\left(F\cap R\left(T\right)\right) $
is $FWM$. Hence it
suffices to observe that ${\rm ker}(T)$, being a subspace,  is $FWM$, and that the
class of $FWM$-sets is closed under taking sums.
\end{proof}

\begin{remark}
It is worth mentioning that in general the affine pre-image of a Motzkin decomposable
set need not be Motzkin decomposable. To wit, consider the ice cream cone
$F$ in $\mathbb R^3$ and the mapping $T: (x_1,x_2,x_3) \mapsto (1,x_2,x_3)$, then the linear function
$x_3-x_2$ does not attain its infimum on  $T^{-1}(F)$, which proves that $T^{-1}(F)$ is not Motzkin decomposable.
\end{remark}

\begin{remark}
In Proposition \ref{prop5} we had proved that the affine pre-image $T^{-1}(F)$ of a $FW$-set is 
$FW$ if $F$ is contained in the range of $T$.  A priori this additional range condition cannot be removed, because we have no result which
guarantees that $F \cap {\rm range}(T)$ is still a $FW$-set (if nonempty). As we just saw, this range condition {\em can} be removed
for $FWM$-sets, and also for $qFW$-sets, so these two classes are invariant under affine pre-images without further range restriction.
\\

\noindent
{\bf Open question:} Let $F$ be a {\em FW}-set and $L$ a linear subspace, is $F \cap L$ a {\em FW}-set?
\end{remark}

\begin{remark}
Altogether we have found the class of {\em FWM}-sets to be closed under finite
products, finite intersections, images and pre-images under affine maps. If we call
a set {\em FWMU} if it is a finite union of {\em FWM}-sets, then sets in this class
are still {\em FW}-sets. By De Morgan's law the class {\em FWMU} remains
closed under finite intersections.  The class
{\em FWMU} remains also closed under affine pre-images, because the pre-image of a union coincides with the union
of the pre-images.   Similarly the class $FWMU$ remains closed under affine images.
\end{remark}

{\color{black}
\section{Parabolic sets}

As we have seen in
Theorem \ref{kummer}, the search
for new {\em FW}-sets does not lead very far   beyond
polyhedrality within the Motzkin class, because if a Motzkin set $F = K+D$ is to be $FW$, then its recession cone $D=0^+F$
must already be polyhedral. The question is therefore whether one can
find $FW$-sets  which exhibit non-polyhedral asymptotic behavior, those  then being necessarily outside the Motzkin class.
 The following result  shows that  such {\em FW}-sets 
do indeed exist.

\begin{theorem}
\label{luo}
{\rm (Luo and Zhang \cite{LZ99}).}  {\em
Let $P$ be a  closed convex polyhedron and define
$F=\{x\in P: x^\ttop Qx+q^\top x + c \leq 0\}$, where $Q=Q^\ttop \succeq 0$.  Then $F$ is  a $FW$-set.}
\hfill $\square$
\end{theorem}

The result generalizes the Frank and Wolfe theorem in the following sense:   if we add just one convex quadratic constraint 
$x^\ttop Qx+q^\top x + c \leq 0$ to a linearly constrained quadratic program, then 
finite infima of quadratics are still attained. As example \ref{example_LZ} shows, adding a second convex quadratic constraint already fails.

The question is now can the Luo-Zhang theorem, just like the Frank-and-Wolf theorem,  be extended from polyhedra $P$ to {\em FWM}-sets $F = K+D$?
That means,  if $F = K+D$ is a {\em FWM}-set, and if
$Q = Q^\ttop \succeq 0$, 
will the set $\mathcal F= \{x\in F: x^\ttop Qx + q^\ttop x + c \leq 0\}$ still be a {\em FW}-set ?
We show by way of a counterexample
that  the answer is in the negative. 

\begin{example}
\label{fail}
We consider the cylinder $F=\{(x_1,x_2,x_3,x_4)\in \mathbb R^4: (x_1-1)^2 +x_2^2\leq 1\}$.
Note that $F$ is a {\em FWM}-set, because it can be represented
as $F = K + L$ for the compact convex set $K= \{(x_1,x_2,0,0)\in \mathbb R^4: (x_1-1)^2+x_2^2 \leq 1\}$ and the subspace
$L = \{0\} \times \{0\} \times \mathbb R \times \mathbb R$.

Now we add the convex quadratic constraint $x_3^2 \leq x_4$ to the constraint set $F$, which leads to the set
$$
\mathcal F = \{x\in F: x_3^2 \leq x_4\}=
  \{x \in \mathbb R^4:  (x_1-1)^2+x_2^2 \leq 1, \; x_3^2 \leq x_4\}.
$$
We will show that $\mathcal F$ is no longer a {\em FW}-set.
This means that the extension of Theorem \ref{luo} from polyhedra $P$ to {\em FWM}-sets $F$ fails.

Consider
the quadratic function
$q(x) = x_4x_1 - 2x_2x_3 + 2$. We claim that $q$ is bounded below on $\mathcal F$ by $0$. Indeed, since $x_1 \geq 0$
on the feasible domain $\mathcal F$, we have $x_4x_1 \geq x_3^2x_1$ on the feasible domain, hence
$q(x) \geq x_3^2 x_1 - 2x_2x_3 + 2=q(x_1,x_2,x_3,x_3^2)$, the expression on the right no longer depending on $x_4$. Let us compute the infimum of that expression on $\mathcal F$.
This comes down to globally solving the program
\begin{eqnarray*}
(P) &
\begin{array}{ll}
\mbox{minimize} & x_3^2x_1 - 2x_2x_3 + 2\\
\mbox{subject to} &  (x_1-1)^2+x_2^2 \leq 1
\end{array}
\end{eqnarray*}
and
it is not hard to see that $(P)$ has infimum 0, but that this infimum is not attained. 
(Solve for $x_3$ with fixed $x_1,x_2$ and show that the value at $(x_1,x_2,x_2/x_1)$ goes to $0$
as $x_1 \to 0^+$, $(x_1-1)^2+x_2^2 = 1$, but that 0 is not attained).

Now if $x^k\in \mathcal F$ is a minimizing sequence
for $q$, then $\xi^k:=(x^k_1,x^k_2,x^k_3,(x_3^k)^2)\in \mathcal F$ is also feasible and gives $q(x^k) \geq q(\xi^k)$, so the sequence $\xi^k$ is also
minimizing, showing that the infimum of $q$ on $\mathcal F$ is the same as the infimum of $(P)$, which is zero. But then the infimum of
$q$ on $\mathcal F$ could not be attained, as otherwise the infimum of $(P)$
would also be attained. Indeed, if the infimum of $q$ on $\mathcal F$
is attained at $\bar{x}\in \mathcal F$, then it must also be attained
at $\bar{\xi} = (\bar{x}_1,\bar{x}_2,\bar{x}_3,\bar{x}_3^2)\in \mathcal F$ because $q(\bar{x})\geq q(\bar{\xi})$, and then the infimum of $(P)$
is attained at $(\bar{x}_1,\bar{x}_2,\bar{x}_3)$, contrary to what was shown.
\end{example}

\begin{remark}
We can write the set $\mathcal F$ as $\mathcal F = K' \times F'$, where
$K'=\{(x_1,x_2): (x_1-1)^2 + x_2^2  \leq 1\}$ is  compact convex, and where $F'$ is the Luo-Zhang set
$F' = \{(x_3,x_4): x_3^2 \leq x_4\}$, which by Theorem \ref{luo} is a {\em FW}-set.   This shows that the cross product of a convex
{\em FW}-set (which is not {\em FWM}) and a compact convex set need no longer be a {\em FW}-set.
\end{remark}

\begin{remark}
We can also write $\mathcal F = (K+L) \cap (F+M)$, where $L,M$ are linear subspaces of $\mathbb R^4$. Indeed, $K,L$ are as in Example \ref{fail},
while $F = \{(0,0,x_3,x_4): x_3^2 \leq x_4\}$ and $M =\mathbb R \times \mathbb R \times \{0\} \times \{0\}$.
Here $K+L$ is {\em FWM}, while $F+M$ is a {\em FW}-set by Theorem \ref{luo}. 
\end{remark}

\begin{remark}
Note that  $\mathcal F$ is a $qFW$-set by Proposition \ref{intersection}, see also \cite[Cor. 2]{LZ99}. 
\end{remark}
}

\section*{Acknowledgement}
Helpful discussions with B. Kummer (HU Berlin) and D. Klatte (Z\"urich) are gratefully acknowledged. 
We are indebted to Vera Roshchina (Australia) for having pointed out reference \cite{mirkil}.
J.E. Mart\'inez-Legaz was supported by the MINECO of Spain, Grant MTM2014-59179-C2-2-P, and by the Severo Ochoa Programme for Centres of Excellence
in R\&D [SEV-2015-0563]. He is affiliated with MOVE (Markets, Organizations and Votes in Economics).
D. Noll was supported by Fondation Math\'ematiques Jacques-Hadamard (FMJH) under PGMO Grant
{\em Robust Optimization for Control}.
%

%
%
%

\begin{thebibliography}{99.}%
%
%
%
\bibitem{andronov}
V.G. Andronov, E.G. Belousov and V.M. Shironin. 
On Solvability of the Problem of  Polynomial  Programming (In Russian). 
Izvestija  Akadem.  Nauk  SSSR,  Tekhnicheskaja Kibernetika 4:1982, 194--197, 
translated as News of the Academy of Science of USSR, Dept. of Technical Sciences, Technical Cybernetics.


\bibitem{banks} B. Bank, J. Guddat, D. Klatte, B. Kummer, K. Tammer.
Non-linear parametric optimization.  Birkh\"auser, Basel-Boston-Stuttgart, 1983.

\bibitem{belusov}
E.G. Belousov. Introduction to Convex Analysis and Integer Programming (in Russian). Moscow University Publisher 1977.

\bibitem{klatte} E.G. Belousov, D. Klatte. A Frank-Wolfe theorem for
convex polynomial programs. Comp. Optim. Appl. 22:2002, 37--48.


\bibitem{oettli} E. Blum, W. Oettli. Direct proof of the existence theorem
in quadratic programming, Operations Research 20:1972, 165--167

\bibitem{collatz} L. Collatz, W. Wetterling. Optimization Problems, Springer Verlag 1975.

\bibitem{eaves} B.C. Eaves. On quadratic programming.  Management Sci. 17(11):1971, 698--711.

\bibitem{frank_and_wolfe} M. Frank and P. Wolfe.  An algorithm for quadratic
programming.   Naval Research Logistics Quarterly 3:1956, 95--110.


\bibitem{goberna} M.A. Goberna, E. Gonz\'alez, J.E. Mart\'inez-Legaz, M.I. Todorov.
Motzkin decomposition of closed convex sets.
J Math. Anal. Appl. 364:2010, 209--221. 


\bibitem{IMT14} A.N. Iusem, J.E. Mart\'{\i}nez-Legaz, M.I. Todorov. Motzkin
predecomposable sets. J. Global Optim. 60(4):2014, 635--647.

\bibitem{Kl60} V. Klee. \ Asymptotes and projections of convex sets. Math.
Scand., 8:1960, 356--362.

\bibitem{kummer} B. Kummer. Globale Stabilit\"at quadratischer Optimierungsprobleme.
Wissenschaftliche Zeitschrift der Humboldt-Universit\"at zu Berlin, Math.-Nat. R. XXVI(5): 1977, 565--569.


\bibitem{LZ99} Z.-Q. Luo, S. Zhang. On extensions of the Frank-Wolfe
theorems. Comput. Optim. Appl. 13:1999, 87--110.


\bibitem{JCA}  J.E. Mart\'{\i}nez-Legaz, D. Noll, W. Sosa. Minimization of quadratic functions on convex sets
without asymptotes. Journal of Convex Analysis, to appear. 


\bibitem{mirkil} H. Mirkil.  New characterizations of polyhedral cones.  Can. J. Math.  9:1957, 1--4.


\bibitem{rock} R.T. Rockafellar. Convex Analysis. Princeton University
Press 1970.

\bibitem{schrijver} A. Schrijver. Theory of Linear and Integer Programming. John Wiley \& Sons 1986.

\bibitem{tam} N. N. Tam. Continuity of the optimal value function in
indefinite quadratic programming. J. Global Optim. 23(1):2002, 43--61.




\end{thebibliography}
%

\end{document}